\documentclass[english,11pt]{article}
\usepackage[T1]{fontenc}
\usepackage[latin9]{inputenc}
\usepackage{verbatim}
\usepackage{float}
\usepackage{amsthm}
\usepackage{amsmath}
\usepackage{amssymb}
\usepackage{graphicx}
\usepackage{esint}

\makeatletter

\floatstyle{ruled}
\newfloat{algorithm}{tbp}{loa}
\providecommand{\algorithmname}{Algorithm}
\floatname{algorithm}{\protect\algorithmname}

\theoremstyle{plain}
\newtheorem{thm}{\protect\theoremname}
  \theoremstyle{definition}
  \newtheorem{example}[thm]{\protect\examplename}
  \theoremstyle{definition}
  \newtheorem{defn}[thm]{\protect\definitionname}
  \theoremstyle{remark}
  \newtheorem{rem}[thm]{\protect\remarkname}
  \theoremstyle{definition}
  \newtheorem{problem}[thm]{\protect\problemname}
  \theoremstyle{plain}
  \newtheorem{lem}[thm]{\protect\lemmaname}

\usepackage{fullpage}
\usepackage{authblk}
\usepackage[caption=false,font=footnotesize]{subfig}
\usepackage{url}

\makeatletter
\def\url@leostyle{%
  \@ifundefined{selectfont}{\def\UrlFont{\sf}}{\def\UrlFont{\footnotesize\rmfamily}}}
\makeatother

\usepackage{amsmath}
\usepackage{amsfonts}

\newcommand{\norm}[1]{\left\Vert#1\right\Vert}

\newcommand{\optmin}{\mathrm{min.}}
\newcommand{\optmax}{\mathrm{max.}}
\newcommand{\optst}{\;\mathrm{s.t.}}

\usepackage{xcolor}

\allowdisplaybreaks[3]

\makeatother

\usepackage{babel}
  \providecommand{\definitionname}{Definition}
  \providecommand{\examplename}{Example}
  \providecommand{\lemmaname}{Lemma}
  \providecommand{\problemname}{Problem}
  \providecommand{\remarkname}{Remark}
\providecommand{\theoremname}{Theorem}

\begin{document}
\title{Differentially Private Convex Optimization with Piecewise\\ Affine Objectives}
\author{Shuo Han\thanks{E-mail: hanshuo@seas.upenn.edu}}
\author{Ufuk Topcu} 
\author{George J. Pappas}
\affil{Department of Electrical and Systems Engineering, University of Pennsylvania, Philadelphia, PA 19104}
\maketitle
\begin{abstract}
Differential privacy is a recently proposed notion of privacy that
provides strong privacy guarantees without any assumptions on the
adversary. The paper studies the problem of computing a differentially
private solution to convex optimization problems whose objective function
is piecewise affine. Such problem is motivated by applications in
which the affine functions that define the objective function contain
sensitive user information. We propose several privacy preserving
mechanisms and provide analysis on the trade-offs between optimality
and the level of privacy for these mechanisms. Numerical experiments
are also presented to evaluate their performance in practice. 

\end{abstract}

\section{Introduction}

With the advance in real-time computing and sensor technology, a growing
number of user-based cyber-physical systems start to utilize user
data for more efficient operation. In power systems, for example,
the utility company now has the capability of collecting near real-time
power consumption data from individual households through advanced
metering infrastructures in order to improve the demand forecast accuracy
and facilitate the operation of power plants~\cite{ca_smart_meters_url}.
At the same time, however, individual customer is exposed to the risk
that the utility company or a potential eavesdropper can learn about
information that the customer did not intend to share, which may include
marketable information such as the type of appliances being used or
even sensitive information such as the customer's daily activities.
Concerns on such privacy issues have been raised~\cite{molina2010private}
and start to become one major hindrance to effective user participation~\cite{hoenkamp2011neglected}. 

Unfortunately, it has been long recognized that \emph{ad-hoc} solutions
(e.g., anonymization of user data) are inadequate to guarantee privacy
due to the presence of public side information. This fact has been
demonstrated through various instances such as identification of Netflix
subscribers in the anonymized Netflix prize dataset through linkage
with the Internet Movie Database (IMDb)~\cite{narayanan2007break}.
Providing rigorous solutions to preserving privacy has become an active
area of research. In the field of systems engineering, recent work
on privacy includes, among others, privacy-preserving filtering of
streaming data~\cite{leny2014differentially}, privacy in smart metering~\cite{sankar2013smart},
privacy in traffic monitoring~\cite{canepa2013framework}, privacy
in stochastic control~\cite{venkitasubramaniam2013privacy}, etc.

Recently, the notion of \emph{differential privacy} proposed by Dwork
and her collaborators has received attention due to its strong privacy
guarantees~\cite{dwork2006calibrating}. The original setting assumes
that the sensitive database is held by a trustworthy party (often
called \emph{curator} in related literature), and the curator needs
to answer external queries (about the sensitive database) that potentially
come from an adversary who is interested in learning information belonging
to some user. Informally, preserving differential privacy requires
that the curator must ensure that the results of the queries remain
approximately unchanged if data belonging to any single user in the
database are modified or removed. In other words, the adversary knows
little about any single user's information from the results of queries.
Interested readers can refer to recent survey papers on differential
privacy for more details on this topic~\cite{dwork2008differential}.

Aside from privacy, another important aspect to consider is the usefulness
of the results of queries. In the context of systems operation, user
data are often used for guiding decisions that optimize systems performance.
Specifically, the ``query'' now becomes the \emph{solution} to the
optimization problem, whereas ``user data'' correspond to \emph{parameters}
that appear in the objective function and/or constraints of the optimization
problem. It is conceivable that preserving user privacy will come
at the cost of optimality. Indeed, without any considerations on systems
performance, one could protect privacy by choosing to ignore user
data, which may lead to solutions that are far from being optimal. 

Several researchers have looked into the application of differential
privacy to optimization problems. For example, Gupta et al. have studied
differential privacy in combinatorial optimization problems and derived
information-theoretic bounds on the utility for a given privacy level~\cite{gupta2010differentially}.
Among all related efforts, one that receives increasingly more attention
is applying differential privacy to convex optimization problems.
Convex optimization problems have traditionally been extensively studied
due to the richness in related results in optimization theory and
their broad applications. In the case of unconstrained convex optimization,
which appears frequently in machine learning (e.g., regression problems),
techniques such as output perturbation and objective perturbation
have been proposed by, among others, Chaudhuri et al.~\cite{chaudhuri2011differentially}
and Kifer et al.~\cite{kifer2012private}. Huang et al. have studied
the setting of private distributed convex optimization, where the
cost function of each agent is considered private~\cite{huang2014differentially}.
Very recently, Hsu et al. have proposed mechanisms for solving linear
programs privately using a differentially private variant of the multiplicative
weights algorithm~\cite{hsu2014privately}.

Rather than focusing on general convex optimization problems or even
linear programs, the work in this paper studies the class of convex
optimization problems whose objective function is piecewise affine,
with the possibility of including linear inequality constraints. This
form of optimization problems arises in applications such as $\ell_{1}$/$\ell_{\infty}$-norm
optimization and resource allocation problems. On one hand, focusing
on this particular class of problems allows us to exploit special
structures that may lead to better algorithms. On the other hand,
such problems can be viewed as a special form of linear programs,
and it is expected that studies on this problem may lead to insights
into applying differential privacy to more general linear programs.

Our major result in this paper is the introduction and analysis of
several mechanisms that preserve differential privacy for convex optimization
problems of this kind. These mechanisms include generic mechanisms
such as the \emph{Laplace mechanism} and the \emph{exponential mechanism}.
We also propose a new mechanism named \emph{differentially private
subgradient method}, which obtains a differentially private solution
by iteratively solving the problem privately. In addition, we provide
theoretical analysis on the suboptimality of these mechanisms and
show the trade-offs between optimality and privacy. %

\section{Problem statement\label{sec:Differential-privacy}}

\subsection{Differential privacy}

Denote by~$\mathcal{D}$ the universe of all databases of interest.
The information that we would like to obtain from a database~$D\in\mathcal{D}$
is represented by a mapping called \emph{query}~$q\colon\mathcal{D}\to\mathcal{Q}$
for some target domain~$\mathcal{Q}$. When the database~$D$ contains
private user information, directly making~$q(D)$ available to the
public may cause users in the database to lose their privacy, and
addition processing (called a \emph{mechanism}) that depends on~$q$
is generally necessary in order to preserve privacy. 
\begin{example}
For a database containing the salaries of a group of people, we can
define~$D=\{x_{i}\}_{i=1}^{n}$, where~$x_{i}\in\mathbb{R}_{+}$
is the salary of user~$i$ (assuming no minimum denomination). Suppose
someone is interested in the average salary of people in the database.
Then the query can be written as~$q(D)=\sum_{i=1}^{n}x_{i}/n$ for
the target domain~$\mathcal{Q}=\mathbb{R}_{+}$.
\end{example}
The fundamental idea of differential privacy is to translate \emph{privacy}
of an individual user in the database into \emph{changes} in the database
caused by that user (hence the name \emph{differential}). With this
connection, preserving privacy becomes equivalent to hiding changes
in the database. Basic changes include addition, removal, and modification
of a single user's data in the database: addition/removal is often
used if privacy is the presence/participation of any single user in
the database (which is common in surveys of diseases); modification
is often used if privacy is in the user data record (if an adversary
cannot tell whether the data record of any single user is modified,
it is impossible for the adversary to obtain exact value of the data).
More generally, changes in database can be defined by a symmetric
binary relation on~$\mathcal{D}\times\mathcal{D}$ called \emph{adjacency}
relation, which is denoted by~$\mathrm{Adj}(\cdot,\cdot)$. It can
be verified that addition, removal, or modification of data belonging
to a single user defines a valid adjacency relation. We will use the
the notion of adjacent database hereafter.

In the framework of differential privacy, all mechanisms under consideration
are \emph{randomized}, i.e., for a given database, the output of such
a mechanism obeys a probability distribution. A differentially private
mechanism must ensure that its output distribution does not vary much
between two adjacent databases.
\begin{defn}
[Differential privacy~\cite{dwork2006calibrating}]A randomized mechanism~$M\colon\mathcal{D}\to\mathcal{Q}$
preserves $\epsilon$-differential privacy if for all~$\mathcal{R}\subseteq\mathcal{Q}$
and all pairs of adjacent databases~$D$ and~$D'$:
\[
\mathbb{P}(M(D)\in\mathcal{R})\leq e^{\epsilon}\mathbb{P}(M(D')\in\mathcal{R}).
\]

\end{defn}
The constant~$\epsilon>0$ indicates the level of privacy: smaller~$\epsilon$
implies higher level of privacy. The notion of differential privacy
promises that an adversary cannot tell from the output with high probability
whether data corresponding to a single user in the database have changed.
This essentially hides user information at the individual level, no
matter what side information an adversary may have. The necessity
of randomized mechanisms is evident from the definition, since the
output of any non-constant deterministic mechanism will normally change
with the input database.
\begin{rem}
One useful interpretation of differential privacy can be made in the
context of detection theory~\cite{wasserman2010statistical,geng2012optimal}.
Suppose privacy is defined as user participation in the database and
thus the adjacency relation is defined as addition/removal of a single
user to/from the database. Consider an adversary who tries to detect
whether a particular user~$i$ is in the database from the output
of an~$\epsilon$-differentially private mechanism~$M$ using the
following rule: report true if the output of~$M$ lies in some set~$\mathcal{R}^{*}\subseteq\mathcal{Q}$
and false otherwise. Let~$D$ be the database with user~$i$ and~$D'$
be the one without. We are interested in the probabilities of two
types of detection errors: false positive probability~$p_{\mathrm{FP}}=\mathbb{P}(M(D')\in\mathcal{R}^{*})$
(i.e., user~$i$ is not present, but the detection algorithm reports
true) and false negative probability~$p_{\mathrm{FN}}=\mathbb{P}(M(D)\notin\mathcal{R}^{*})=\mathbb{P}(M(D)\in\mathcal{Q}\backslash\mathcal{R}^{*})$,
both of which need to be small for achieving good detection. Since~$D$
and $D'$ are adjacent, we know from the definition of differential
privacy that
\begin{align*}
\mathbb{P}(M(D)\in\mathcal{R}^{*}) & \leq e^{\epsilon}\mathbb{P}(M(D')\in\mathcal{R}^{*}),\\
\mathbb{P}(M(D')\in\mathcal{Q}\backslash\mathcal{R}^{*}) & \leq e^{\epsilon}\mathbb{P}(M(D)\in\mathcal{Q}\backslash\mathcal{R}^{*}),
\end{align*}
which lead to
\begin{equation}
p_{\mathrm{FN}}+e^{\epsilon}p_{\mathrm{FP}}\geq1\quad\text{and}\quad e^{\epsilon}p_{\mathrm{FN}}+p_{\mathrm{FP}}\geq1.\label{eq:privacy_in_detection}
\end{equation}
The conditions in~(\ref{eq:privacy_in_detection}) imply that~$p_{\mathrm{FN}}$
and~$p_{\mathrm{FP}}$ cannot be both too small. Namely, these conditions
limit the detection capability of the adversary so that the privacy
of user~$i$ is protected. For example, if~$\epsilon=0.1$ and the
false negative probability $p_{\mathrm{FN}}=0.05$, then the false
positive probability~$p_{\mathrm{FP}}\geq\max\{1-e^{\epsilon}p_{\mathrm{FN}},e^{-\epsilon}(1-p_{\mathrm{FN}})\}\approx0.94$,
which is quite large.
\end{rem}

\subsection{Problem statement}

We consider minimization problems whose objective function~$f\colon\mathbb{R}^{d}\to\mathbb{R}$
is convex and piecewise affine:
\begin{equation}
f(x)=\max_{i=1,2,\dots,m}\{a_{i}^{T}x+b_{i}\}\label{eq:pwa}
\end{equation}
for some constants~$\{a_{i}\in\mathbb{R}^{d},b_{i}\in\mathbb{R}\}_{i=1}^{m}$.
For generality, we also add additional linear inequality constraints
that define a convex polytope~$\mathcal{P}$, so that the optimization
problem has the following form: 
\begin{alignat}{2}
 & \underset{x}{\optmin}\quad &  & f(x)\qquad\optst\quad x\in\mathcal{P}.\label{eq:pwa_problem}
\end{alignat}

In this paper, we restrict our attention to the case where user information
is in $\{b_{i}\}_{i=1}^{m}$, so that the database~$D=\{b_{i}\}_{i=1}^{m}$.
Any other information, including $\{a_{i}\}_{i=1}^{m}$ and~$\mathcal{P}$,
is considered as public and fixed. Define the adjacency relation between
two databases~$D=\{b_{i}\}_{i=1}^{m}$ and $D'=\{b_{i}'\}_{i=1}^{m}$
as follows: 
\begin{equation}
\mathrm{Adj}(D,D')\quad\text{if and only if}\quad\max_{i\in\{1,2,\dots,m\}}|b_{i}-b_{i}'|\leq b_{\mathrm{max}}.\label{eq:def_adj}
\end{equation}
Since~$D$ gives a complete description of problem~(\ref{eq:pwa_problem}),
we will often use~$D$ to represent both the database and the corresponding
optimization problem. With the definition of adjacency relation, we
are ready to give the formal problem statement.
\begin{problem}
\label{prob:private_pwa}For all problems in the form of~(\ref{eq:pwa_problem}),
find a mechanism~$M$ that outputs an approximate optimal solution
that preserves $\epsilon$-differential privacy under the adjacency
relation~(\ref{eq:def_adj}). Namely, for all $\mathcal{R}\subseteq\mathcal{P}$
and all adjacent databases~$D$ and $D'$, the mechanism~$M$ must
satisfy 
\[
\mathbb{P}(M(D)\in\mathcal{R})\leq e^{\epsilon}\mathbb{P}(M(D')\in\mathcal{R}).
\]

\end{problem}

\subsection{Convex problems with piecewise affine objectives}

In the following, we give several examples of convex minimization
problems whose objective is piecewise affine:
\begin{example}
[$\ell_\infty$-norm]The $\ell_{\infty}$-norm $f(x)=\norm{x}_{\infty}$
can be rewritten in the form of~(\ref{eq:pwa}) consisting of~$2d$
affine functions: 
\[
f(x)=\max_{i=1,2,\dots,d}|x_{i}|=\max\{x_{1},-x_{1},x_{2},-x_{2},\dots,x_{d},-x_{d}\}.
\]

\end{example}
\begin{example}
[$\ell_1$-norm]The $\ell_{1}$-norm $f(x)=\norm{x}_{1}$ can be rewritten
in the form of~(\ref{eq:pwa}) consisting of~$2^{d}$ affine functions:
\[
f(x)=\sum_{i=1}^{d}\max\{x_{i},-x_{i}\}=\max_{\{\alpha_{i}\in\{0,1\}\}_{i=1}^{d}}\sum_{i=1}^{d}(-1)^{\alpha_{i}}x_{i}.
\]

\end{example}
\begin{example}
[Resource allocation]Consider the following resource allocation problem,
which is one such example where \emph{private optimal solution} may
be desired. Suppose we need to purchase a certain kind of resource
and allocate it among~$n$ agents, and we need to decide the optimal
amount of resource to purchase. Agent~$i$, if being allocated $z_{i}$
amount of resource, can provide utility~$c_{i}z_{i}$, where~$c_{i}$
is its utility gain. This holds until the maximum desired resource
(denoted by~$\bar{z}_{i}$) for agent~$i$ is reached. 

Suppose the total amount of resource to allocate is given as $x\geq0$.
The maximum utility gain can be determined by the optimal value of
the following optimization problem
\begin{alignat}{2}
 & \underset{z}{\optmax}\quad &  & c^{T}z\qquad\optst\quad\mathbf{1}^{T}z\leq x,\qquad0\preceq z\preceq\bar{z},\label{eq:RA_primal}
\end{alignat}
whose optimal value is denoted as~$U(x)$. One can show that $U(x)$
is a concave and piecewise affine function by considering the dual
of~problem (\ref{eq:RA_primal}):
\begin{alignat}{2}
 & \underset{\lambda,\nu}{\optmin}\quad &  & \nu x+\lambda^{T}\bar{z}\label{eq:RA_dual}\\
 & \optst\quad &  & \nu\geq0,\qquad\lambda\succeq0,\qquad\lambda+\nu\mathbf{1}-c\succeq0.\nonumber 
\end{alignat}
Strong duality holds since the primal problem~(\ref{eq:RA_primal})
is always feasible ($z=0$ is a feasible solution), which allows us
to redefine $U(x)$ as the optimal value of problem~(\ref{eq:RA_primal}).
In addition, since the optimal value of any linear program can always
be attained at a vertex of the constraint polytope, we can rewrite
$U$ as the pointwise minimum of affine functions (hence~$U$ is
concave): 
\begin{equation}
U(x)=\min_{i=1,2,\dots,m}\{\nu_{i}x+\lambda_{i}^{T}\bar{z}\},\label{eq:pwa_RA}
\end{equation}
where $\{(\nu_{i},\lambda_{i})\}_{i=1}^{m}$ are the vertices of the
constraint polytope in problem~(\ref{eq:RA_dual}). If we are interested
in maximizing the net utility $U(x)-\mu x$ over $x$, where~$\mu$
is the price of the resource, the problem becomes equivalent to one
in the form of~(\ref{eq:pwa_problem}). \end{example}
\begin{rem}
In certain applications, the maximum desired resource~$\bar{z}_{i}$,
which is present in the affine functions in~(\ref{eq:pwa_RA}), may
be treated as private information by agent~$i$. As an example, consider
the scenario where each agent represents a consumer in a power network,
and the resource to be allocated is the available electricity. Then
the maximum desired resource~$\bar{z}_{i}$ can potentially reveal
activities of consumer~$i$, e.g., small~$\bar{z}_{i}$ may indicate
that consumer~$i$ is away from home.\end{rem}

\section{Useful tools in differential privacy}

This section reviews several useful tools in differential privacy
that will be used in later sections. Material in this section includes
the (vector) Laplace mechanism, the exponential mechanism, the post-processing
rule, and composition of private mechanisms. Readers who are familiar
with these topics may skip this section.

When the range of query~$\mathcal{Q}$ is~$\mathbb{R}$, one commonly
used differentially private mechanism is the \emph{Laplace mechanism}~\cite{dwork2006calibrating}.
In this paper, we use a multidimensional generalization of the Laplace
mechanism for queries that lie in $\mathbb{R}^{d}$. Suppose the sensitivity
of query~$q$, defined as 
\[
\Delta:=\max_{D,D'}\norm{q(D)-q(D')}_{\infty},
\]
is bounded. Then one way to achieve $\epsilon$-differential privacy
is to add i.i.d. Laplace noise $\mathrm{Lap}(d\Delta/\epsilon)$ to
each component of~$q$, which is guaranteed by the sequential composition
theorem (Theorem~\ref{thm:seq_composition}) listed at the end of
this section. However, a similar mechanism that requires less noise
can be adopted in this case by using the fact that the $\ell_{2}$-sensitivity
of the query~$\Delta_{2}$ (defined below) is also bounded:
\[
\Delta_{2}:=\max_{D,D'}\norm{q(D)-q(D')}_{2}\leq\sqrt{d}\Delta.
\]

\begin{thm}
\label{thm:vector_laplace}For a given query~$q$, let $\Delta_{2}=\max_{D,D'}\norm{q(D)-q(D')}_{2}$
be the $\ell_{2}$-sensitivity of~$q$. Then the mechanism~$M(D)=q(D)+w$,
where $w$ is a random vector whose probability distribution is proportional
to~$\exp(-\epsilon\norm{w}_{2}/\Delta_{2})$, preserves $\epsilon$-differential
privacy.
\end{thm}
We are not aware of the name of the mechanism described in Theorem~\ref{thm:vector_laplace}.
Although the additive perturbation~$w$ in Theorem~\ref{thm:vector_laplace}
does not follow the Laplace distribution (in fact, it follows the
Gamma distribution), we will still refer to this mechanism as the
\emph{vector} Laplace mechanism due to its close resemblance to the
(scalar) Laplace mechanism.

Another useful and quite general mechanism is the \emph{exponential
mechanism.} This mechanism requires a scoring function $u\colon\mathcal{Q}\times\mathcal{D}\to\mathbb{R}$.
For a given database~$D$, the scoring function~$u$ characterizes
the quality of any candidate query: if a query~$q$ is more desirable
than another query~$q'$, then we have $u(D,q)>u(D,q')$. The exponential
mechanism~$M_{E}(D;u)$ guarantees $\epsilon$-differential privacy
by randomly reporting~$q$ according to the probability density function
\[
\frac{\exp(\epsilon u(D,q)/2\Delta_{u})}{\int_{q'\in\mathcal{Q}}\exp(\epsilon u(D,q')/2\Delta_{u})\, dq'},
\]
where
\[
\Delta_{u}:=\max_{x}\max_{D,D'\colon\mathrm{Adj}(D,D')}|u(D,x)-u(D',x)|
\]
is the (global) sensitivity of the scoring function~$u$. 
\begin{thm}
[McSherry and Talwar~\cite{mcsherry2007mechanism}]\label{thm:exp_mech}The
exponential mechanism is $\epsilon$-differentially private.
\end{thm}
When the range~$\mathcal{Q}$ is finite, i.e., $|\mathcal{Q}|<\infty$,
the exponential mechanism has the following probabilistic guarantee
on the suboptimality with respect to the scoring function.
\begin{thm}
[McSherry and Talwar~\cite{mcsherry2007mechanism}]\label{thm:exp_subopt}Consider
the exponential mechanism~$M_{E}(D;u)$ acting on a database D under
a scoring function~$u$. If~$\mathcal{Q}$ is finite, i.e., $|\mathcal{Q}|<\infty$,
then~$M_{E}$ satisfies
\[
\mathbb{P}\left[u_{\mathrm{opt}}-u(D,M_{E}(D;u))\geq\frac{2\Delta_{u}}{\epsilon}(\log|\mathcal{Q}|+t)\right]\leq e^{-t},
\]
where~$u_{\mathrm{opt}}=\max_{q\in\mathcal{Q}}u(D,q)$.
\end{thm}
It is also possible to obtain the expected suboptimality using the
following lemma.
\begin{lem}
\label{lem:mean_with_exp_bound}Suppose a random variable~$X$ satisfies:
(1)$X\geq0$ and (2) $\mathbb{P}(X\geq t)\leq e^{-\alpha t}$ for
some $\alpha>0$. Then it holds that $\mathbb{E}[X]\leq1/\alpha$.\end{lem}
\begin{proof}
Use the fact that~$X\geq0$ to write~$X=\int_{0}^{\infty}I(X\geq t)\, dt.$
Then
\[
\mathbb{E}[X]=\mathbb{E}\left[\int_{0}^{\infty}I(X\geq t)\, dt\right]=\int_{0}^{\infty}\mathbb{E}[I(X\geq t)]\, dt=\int_{0}^{\infty}\mathbb{P}(X\geq t)\, dt\leq\int_{0}^{\infty}e^{-\alpha t}\, dt=1/\alpha.
\]

\end{proof}
Combine Theorem~\ref{thm:exp_subopt} and Lemma~\ref{lem:mean_with_exp_bound}
to obtain the expected suboptimality.
\begin{thm}
\label{thm:exp_mean_subopt}Under the same assumptions in Theorem~\ref{thm:exp_subopt},
the exponential mechanism~$M_{E}(D;u)$ satisfies
\[
\mathbb{E}\left[u_{\mathrm{opt}}-u(D,M_{E}(D;u))\right]\leq2\Delta_{u}(1+\log|\mathcal{Q}|)/\epsilon.
\]

\end{thm}
Finally, there are two very useful theorems that enable construction
of new differentially private mechanisms from existing ones. 
\begin{thm}
[Post-processing]\label{thm:post-processing}Suppose a mechanism~$M\colon\mathcal{D}\to\mathcal{Q}$
preserves $\epsilon$-differential privacy. Then for any function~$f$,
the (functional) composition $f\circ M$ also preserves $\epsilon$-differential
privacy.
\end{thm}
\begin{thm}
[Seqential composition~\cite{mcsherry2009privacy}]\label{thm:seq_composition}Suppose
a mechanism~$M_{1}$ preserves $\epsilon_{1}$-differential privacy,
and another mechanism~$M_{2}$ preserves $\epsilon_{2}$-differential
privacy. Define a new mechanism~$M(D):=(M_{1}(D),M_{2}(D))$. Then
the mechanism~$M$ preserves $(\epsilon_{1}+\epsilon_{2})$-differential
privacy.\end{thm}

\section{Privacy-preserving mechanisms\label{sec:Privacy-preserving-mechanisms}}

This section presents the main theoretical results of this paper.
In particular, we propose several mechanisms that are able to obtain
a differentially private solution to Problem~\ref{prob:private_pwa}.
We also give suboptimality analysis for most mechanisms and show the
trade-offs between optimality and privacy.

\subsection{The Laplace mechanism acting on the problem data}

One straightforward way of preserving differential privacy is to obtain
the optimal solution from a privatized version of problem~(\ref{eq:pwa_problem})
by publishing the \emph{entire} database~$D$ privately using the
vector Laplace mechanism described in Theorem~\ref{thm:vector_laplace}.
Privacy is guaranteed by the post-processing rule: once the problem
is privatized, obtaining the optimal solution can be viewed as post-processing
and does not change the level of privacy due to Theorem~\ref{thm:post-processing}. 
\begin{thm}
The mechanism that outputs $M_{P}(b)=b+w_{P}$, where~$w_{P}$ is
drawn from the probability density function that is proportional to~$\exp(-\epsilon\norm{w_{P}}_{2}/\sqrt{m}b_{\max})$,
is $\epsilon$-differentially private.\end{thm}
\begin{proof}
In this case, the query is $b$, whose $\ell_{2}$-sensitivity can
be obtained as $\Delta=\max_{b,b'}\norm{b-b'}_{2}=\sqrt{m}b_{\mathrm{max}}.$
Combining with Theorem~\ref{thm:vector_laplace} completes the proof.
\end{proof}

\subsection{The Laplace mechanism acting on the problem solution}

Another way of preserving differential privacy is to apply the vector
Laplace mechanism directly on the optimal solution $x_{\mathrm{opt}}(D)$
of the problem: $M_{S}(D)=x_{\mathrm{opt}}(D)+w_{S}$. The additive
noise $w_{S}$ is drawn from the distribution proportional to~$\exp(-\epsilon\norm{w_{S}}_{2}/\sqrt{d}\Delta)$,
where~$\Delta$ is the sensitivity of the optimal solution, i.e.,
\[
\Delta=\max_{D,D'\colon\mathrm{Adj}(D,D')}\norm{x_{\mathrm{opt}}(D)-x_{\mathrm{opt}}(D')}_{2}.
\]
This mechanism is $\epsilon$-differentially private also due to Theorem~\ref{thm:vector_laplace}.

Unfortunately, it is generally difficult to analyze how the optimal
solution~$x_{\mathrm{opt}}(D)$ changes with~$D$, and hence the
exact value of~$\Delta$ is often unavailable. However, when the
set $\mathcal{P}$ is compact, an upper bound of~$\Delta$ can be
given by the diameter of $\mathcal{P}$, defined as $\mathrm{diam}(\mathcal{P}):=\max_{x,y\in\mathcal{P}}\norm{x-y}_{2}.$
Although~$\mathrm{diam}(\mathcal{P})$ is still difficult to compute
for a generic set~$\mathcal{P}$, there are several cases where its
exact value or an upper bound can be computed efficiently. One simple
case is when $\mathcal{P}=\{x\colon0\leq x_{i}\leq1,\: i=1,2,\dots,d\}$
is a hypercube and hence~$\mathrm{diam}(\mathcal{P})=\sqrt{d}$.
In the more general case where~$\mathcal{P}$ is described by a set
of linear inequalities, an upper bound can be obtained by computing
the longest axis of the Löwner-John ellipsoid of~$\mathcal{P}$,
i.e., the minimum-volume ellipsoid that covers~$\mathcal{P}$. The
Löwner-John ellipsoid can be approximated from the maximum-volume
inscribed ellipsoid, which can be obtained by solving a convex optimization
problem (in particular, a semidefinite problem, cf.~\cite[page 414]{boyd2004convex}
).

Suboptimality analysis for this mechanism is given by the following
theorem.
\begin{thm}
\label{thm:solution_subopt}Define~$G=\max_{i}\norm{a_{i}}_{2}$.
The expected suboptimality for the solution perturbation mechanism~$M_{S}$
is bounded as 
\[
\mathbb{E}[f(M_{S}(D))-f(x_{\mathrm{opt}})]\leq Gd^{3/2}\Delta/\epsilon.
\]
\end{thm}
\begin{proof}
Since $f(x)-f(x_{\mathrm{opt}})\geq0$ for all $x\in\mathcal{P}$,
we have 
\[
\mathbb{E}[f(M_{S}(D))-f(x_{\mathrm{opt}})]=\mathbb{E}|f(M_{S}(D))-f(x_{\mathrm{opt}})|.
\]
It is not difficult to show that~$f$ is Lipschitz with~$G$ as
the Lipschitz constant, i.e., $|{\normalcolor f(x)-f(y)|}\leq G\norm{x-y}_{2},$which
leads to
\[
\mathbb{E}[f(M_{S}(D))-f(x_{\mathrm{opt}})]\leq G\mathbb{E}\norm{M_{S}(D)-x_{\mathrm{opt}}}_{2}=G\mathbb{E}\norm{w_{S}}_{2}=Gd\cdot\sqrt{d}\Delta/\epsilon=Gd^{3/2}\Delta/\epsilon.
\]

\end{proof}
Theorem~\ref{thm:solution_subopt} shows that the expected suboptimality
grows as~$\epsilon$ decreases (i.e., the level of privacy increases).
The suboptimality also grows with~$d$, which is the dimension of
the decision variable~$x$.

\subsection{The exponential mechanism}

To use the exponential mechanism for privately solving minimization
problems, one natural choice of the scoring function is the negative
objective function $-f$. However, this choice may not work in all
cases, since changes in user data can lead to an infeasible problem,
which yields unbounded sensitivity. Even when the problem remains
feasible, the sensitivity of the objective function with respect to
data can be difficult to compute for a generic optimization problem.
Nevertheless, the following shows that the sensitivity for our problem
is bounded and can be computed exactly.
\begin{lem}
\label{lem:sensitivity_exp_mech}Suppose the scoring function is given
as 
\begin{align*}
u(x,D) & =-f(x,D)=-\max_{i=1,2,\dots,m}\{a_{i}^{T}x+b_{i}\}.
\end{align*}
Then the sensitivity of~$u$ for the adjacency relation defined in~(\ref{eq:def_adj})
is~$\Delta_{u}=b_{\max}$.\end{lem}
\begin{proof}
See Appendix~\ref{sub:proof_sensitivity_exp_mech}.
\end{proof}
As a result of Theorem~\ref{thm:exp_mech} and Lemma~\ref{lem:sensitivity_exp_mech},
we know that we can achieve $\epsilon$-differential privacy by using
the exponential mechanism given in the following theorem.
\begin{thm}
The exponential mechanism~$M_{E}$, which randomly outputs $\tilde{x}_{\mathrm{opt}}$
according the probability density function
\begin{equation}
\frac{\exp(-\epsilon f(\tilde{x}_{\mathrm{opt}},D)/2b_{\max})}{\int_{x\in\mathcal{P}}\exp(-\epsilon f(x,D)/2b_{\max})\, dx},\label{eq:exp_dist_f}
\end{equation}
is $\epsilon$-differentially private.\end{thm}
\begin{rem}
The denominator in~(\ref{eq:exp_dist_f}) needs to remain bounded
in order for~(\ref{eq:exp_dist_f}) to be a valid probability distribution.
This trivially holds when $\mathcal{P}$ is bounded. When $\mathcal{P}$
is unbounded, this can be shown by using the fact that $f(x,D)$ is
affine in $x$ and the integrand decreases exponentially fast as $\norm{x}\to\infty$. 
\end{rem}
Suboptimality analysis for the exponential mechanism is given by the
following theorem.
\begin{thm}
\label{thm:exp_pwa_subopt}The expected suboptimality for the exponential
mechanism $M_{E}$ is bounded as
\[
\mathbb{E}[f(M_{E}(D))-f_{\mathrm{opt}}]\leq C(0,\epsilon)\cdot2b_{\max}/\epsilon,
\]
where $f_{\mathrm{opt}}=\min_{x\in\mathcal{P}}f(x,D)$ and for any~$\gamma\geq0$,
\[
C(\gamma,\epsilon)=\frac{\exp(-\epsilon f_{\mathrm{opt}}/2b_{\max})\int_{x\colon f(x,D)\geq f_{\mathrm{opt}}+\gamma}\, dx}{\int_{x\in\mathcal{P}}\exp(-\epsilon f(x,D)/2b_{\max})\, dx}.
\]
\end{thm}
\begin{proof}
We first prove that for any~$\gamma\geq0$, 
\begin{equation}
\mathbb{P}[f(M_{E}(D))-f_{\mathrm{opt}}\geq\gamma]\leq C(\gamma,\epsilon)\exp(-\epsilon\gamma/2b_{\max}).\label{eq:prob_exp_mech}
\end{equation}
For any given $a\in\mathbb{R}$, the exponential mechanism~$M_{E}(D)$
with scoring function $u$ satisfies 
\begin{align*}
\mathbb{P}[u(M_{E}(D))\leq a] & =\frac{\int_{x:u(x,D)\leq a}\exp(\epsilon u(x,D)/2b_{\max})\, dx}{\int_{x\in\mathcal{P}}\exp(\epsilon u(x,D)/2b_{\max})\, dx}\\
 & \leq\frac{\exp(\epsilon a/2b_{\max})\int_{x:u(x,D)\leq a}\, dx}{\int_{x\in\mathcal{P}}\exp(\epsilon u(x,D)/2b_{\max})\, dx}.
\end{align*}
Set $u(x,D)=-f(x,D)$ and $a=-(\gamma+f_{\mathrm{opt}})$ to obtain~(\ref{eq:prob_exp_mech}).

Note that $C(\gamma,\epsilon)\leq C(0,\epsilon)$ for all $\gamma\geq0$.
Then
\begin{equation}
\mathbb{P}[f(M_{E}(D))-f_{\mathrm{opt}}\geq\gamma]\leq C(0,\epsilon)\exp(-\epsilon\gamma/2b_{\max}).\label{eq:prob_exp_mech-2}
\end{equation}
Apply Lemma~\ref{lem:mean_with_exp_bound} on~(\ref{eq:prob_exp_mech-2})
to complete the proof. 
\end{proof}
It can be shown that~$C(0,\epsilon)$ increases as~$\epsilon$ decreases.
Therefore, similar to the solution perturbation mechanism~$M_{S}$
described in Theorem~\ref{thm:solution_subopt}, the expected suboptimality
of~$M_{E}$ grows as~$\epsilon$ decreases.

\subsection{Private subgradient method}

If privacy is not a concern, one way to solve the optimization problem~(\ref{eq:pwa_problem})
is to use the subgradient method, which iteratively searches for an
optimal solution by moving along the direction of a subgradient. Although
the direction of a subgradient is not necessarily a descent direction,
the subgradient method is still guaranteed to converge if one keeps
track of the best (rather than the most recent) solution among all
past iterations. Recall that $g$ is a subgradient of~$f$ at~$x_{0}$
if and only if for all $x$: 
\begin{equation}
f(x)\geq f(x_{0})+g^{T}(x-x_{0}).\label{eq:subgrad_cond}
\end{equation}
For a convex and piecewise affine function~$f$, its subgradient
at any given $x_{0}$ can be obtained as follows. First find $k\in\{1,2,\dots,m\}$
such that
\begin{equation}
a_{k}^{T}x_{0}+b_{k}=\max_{i=1,2,\cdots,m}\{a_{i}^{T}x_{0}+b_{i}\}.\label{eq:compute_subgrad}
\end{equation}
Then a subgradient at~$x_{0}$ is $a_{k}$, which can be verified
using the definition~(\ref{eq:subgrad_cond}).

It can be seen from~(\ref{eq:compute_subgrad}) that computing subgradients
requires access to the private data~$\{b_{i}\}_{i=1}^{m}$. Following
from Hsu et al.~\cite{hsu2014privately}, in order to preserve privacy
when applying any iterative method (such as the subgradient method),
one must make sure to: (1) privatize the computation during each iteration;
(2) limit the total number of iterations. 

One method for obtaining a subgradient privately is to perturb the
true subgradient by adding, e.g., Laplace noise~\cite{song2013stochastic}.
In our case, since the candidate subgradients come from a finite set~$\{a_{i}\}_{i=1}^{m}$,
we propose to use the exponential mechanism to privatize the computation
of subgradients. Choose the scoring function~$u_{\mathrm{sub}}\colon\{1,2,\cdots,m\}\to\mathbb{R}$
as $u_{\mathrm{sub}}(i;x,D)=a_{i}^{T}x+b_{i}$ (in the following,
we will sometimes drop~$D$ for conciseness). The sensitivity of
$u_{\mathrm{sub}}$ at any given~$x_{0}$, which is denoted as $\Delta_{u_{\mathrm{sub}}}(x_{0})$,
can be computed as
\[
\max_{i\in\{1,2,\dots,m\}}\max_{D,D'}|u_{\mathrm{sub}}(i;x_{0},D)-u_{\mathrm{sub}}(i;x_{0},D')|=b_{\max}.
\]

\begin{algorithm}[tbph]
\begin{enumerate}
\item Choose the scoring function~$u\colon\{1,2,\cdots,m\}\to\mathbb{R}$
as 
\[
u_{\mathrm{sub}}(i;x_{0})=a_{i}^{T}x_{0}+b_{i}.
\]

\item Select the index $i^{*}$ using the exponential mechanism: 
\[
\mathbb{P}(i^{*}=i)\propto\exp(-\epsilon u_{\mathrm{sub}}(i;x_{0})/2b_{\max}).
\]

\item Output~$a_{i^{*}}$ as the approximate subgradient at $x_{0}$.
\end{enumerate}
\caption{$\epsilon$-differentially private subgradient.\label{alg:dp_subgrad}}

\end{algorithm}
If the subgradient computation in the regular subgradient method is
replaced by Algorithm~\ref{alg:dp_subgrad}, the modified subgradient
method (Algorithm~\ref{alg:dp_subgrad-1}) can be shown to preserve
$\epsilon$-differential privacy using the sequential composition
theorem, since each iteration preserves~$(\epsilon/k)$-differential
privacy and the total number of iterations is~$k$.

\begin{algorithm}[tbph]
\begin{enumerate}
\item Choose the number of iterations~$k$, step sizes~$\{\alpha_{i}\}_{i=1}^{k}$,
and~$x^{(1)}\in\mathcal{P}$.
\item For $i=1,2,\dots,k$, repeat:

\begin{enumerate}
\item Obtain an $(\epsilon/k$)-private subgradient $g^{(i)}$ using Algorithm~\ref{alg:dp_subgrad};
\item Update~$x^{(i+1)}:=x^{(i)}-\alpha_{i}g^{(i)}.$
\end{enumerate}
\item Output~$x^{(k+1)}$ as the solution.
\end{enumerate}
\caption{$\epsilon$-differentially private subgradient method.\label{alg:dp_subgrad-1}}
\end{algorithm}

However, since the output of Algorithm~\ref{alg:dp_subgrad} does
not correspond to a true subgradient, it is natural to ask how this
affects convergence of the optimization procedure. For any given~$x_{0}$,
define
\begin{align*}
\gamma(j;x_{0}) & :=\max_{i\in\{1,2,\dots,m\}}u_{\mathrm{sub}}(i;x_{0})-u_{\mathrm{sub}}(j;x_{0})\\
 & =\max_{i=1,2,\dots,m}\{a_{i}^{T}x_{0}+b_{i}\}-(a_{j}^{T}x_{0}+b_{j}),
\end{align*}
which is the suboptimality gap of choosing~$a_{j}$ (in Algorithm~\ref{alg:dp_subgrad})
measured by the scoring function~$u_{\mathrm{sub}}$. As a consequence
of using the exponential mechanism in step~2 of Algorithm~\ref{alg:dp_subgrad},
an upper bound on the expectation of~$\gamma$ can be obtained using
Theorem~\ref{thm:exp_mean_subopt}:
\begin{equation}
\mathbb{E}_{i^{*}}[\gamma(i^{*};x_{0})]\leq2b_{\max}(1+\log m)/\epsilon.\label{eq:bound_exp_gamma}
\end{equation}
The following lemma shows how the suboptimality gap~$\gamma$ affects
the subgradient condition~(\ref{eq:subgrad_cond}) if~$a_{j}$ is
used as a subgradient.
\begin{lem}
\label{lem:subgrad_violation}For all $x$, it holds that
\[
f(x)\geq f(x_{0})+a_{j}^{T}(x-x_{0})-\gamma(j;x_{0}).
\]
\end{lem}
\begin{proof}
We have
\begin{align*}
f(x) & =\max_{i=1,2,\dots,m}\{a_{i}^{T}x+b_{i}\}\\
 & \geq a_{j}^{T}x+b_{j}\\
 & =(a_{j}^{T}x_{0}+b_{j})+a_{j}^{T}(x-x_{0})\\
 & =\max_{i=1,2,\dots,m}\{a_{i}^{T}x_{0}+b_{i}\}-\gamma(j;x_{0})+a_{j}^{T}(x-x_{0})\\
 & =f(x_{0})+a_{j}^{T}(x-x_{0})-\gamma(j;x_{0}).
\end{align*}
\end{proof}
\begin{rem}
If $j=\arg\max_{i\in\{1,2,\dots,m\}}u_{\mathrm{sub}}(i;x_{0})$, i.e.,
$a_{j}$ is a true subgradient, then Lemma~\ref{lem:subgrad_violation}
recovers the original definition of subgradient:
\[
f(x)\geq f(x_{0})+a_{j}^{T}(x-x_{0}).
\]

\end{rem}
Lemma~\ref{lem:subgrad_violation} shows that the suboptimality gap~$\gamma$
also characterizes the extent that the subgradient condition is violated.
Now we are ready to show the expected suboptimality of the differentially
private subgradient method using both the bound~(\ref{eq:bound_exp_gamma})
on~$\mathbb{E}_{i^{*}}[\gamma(i^{*};x_{0})]$ and Lemma~\ref{lem:subgrad_violation}. 
\begin{thm}
\label{thm:dp_subgrad_subopt}When the $\epsilon$-differentially
private subgradient method is applied, the expected suboptimality
after~$k$ iterations is bounded as
\begin{equation}
\mathbb{E}\left[\min_{i=1,2,\dots,k}f(x^{(i)})-f_{\mathrm{opt}}\right]\leq\frac{R^{2}+G^{2}\sum_{i=1}^{k}\alpha_{i}^{2}}{2\sum_{i=1}^{k}\alpha_{i}}+\bar{\gamma}(\epsilon/k),\label{eq:dp_subgrad_subopt}
\end{equation}
where~$R=\mathrm{diam}(\mathcal{P})$, $G=\max_{i=1,2,\dots,m}\norm{a_{i}}_{2}$,
and~$\bar{\gamma}(z)=2b_{\max}(1+\log m)/z$.\end{thm}
\begin{proof}
See Appendix~\ref{sub:proof_dp_subgrad_subopt}.
\end{proof}
Theorem~\ref{thm:dp_subgrad_subopt} shows a tradeoff between privacy
and suboptimality. The first term, which also appears in the convergence
analysis for the regular subgradient method, implies that the optimal
gap vanishes as the number of iterations~$k\to\infty$. However,
if~$k$ becomes too large, inaccuracy in the private subgradients
will start to act as a dominant factor in suboptimality as the second
term indicates. In particular, Theorem~\ref{thm:dp_subgrad_subopt}
implies that there exists an optimal number of iterations: as the
number of iterations grows, the first term in~(\ref{eq:dp_subgrad_subopt})
decreases, whereas $\bar{\gamma}(\epsilon/k)$ increases due to increased
level of privacy for each iteration. Similar to previous results given
by Theorem~\ref{thm:solution_subopt} and~\ref{thm:exp_pwa_subopt},
the second term (due to privacy) grows as~$\epsilon$ decreases.

\section{Numerical experiments\label{sec:Numerical-experiments}}

\subsection{Implementation details}

In all simulations, the problem data~$\{(a_{i},b_{i})\}_{i=1}^{m}$
are generated from i.i.d. Gaussian distributions. The constraint set
is chosen to be a $d$-dimensional hypercube centered at the origin:
$\mathcal{P}=\{x\colon-c\preceq x\preceq c\},$ whose diameter $\mathrm{diam}(\mathcal{P})=2\sqrt{d}c$.
The level of privacy~$\epsilon$ is set at~$0.1$. The expected
objective value for different privacy-preserving mechanisms is approximated
by the sample average from~$1000$ runs.

\paragraph{Implementation of the vector Laplace noise}

One way to efficiently generate~$w$ from the distribution proportional
to~$\exp(-\lambda\norm{w}_{2})$ is to draw its magnitude~$\bar{w}$
and direction~$\hat{e}$ (as a unit vector) separately. It can be
shown that~$\bar{w}$ follows the Gamma distribution $\Gamma(d,\lambda)$
and the distribution of $\hat{e}$ is isotropic~\cite{chaudhuri2008privacy}.
In order to draw a sample from $\Gamma(d,\lambda)$, one can draw
$d$ i.i.d. samples $w_{1},w_{2},\dots,w_{d}$ from the exponential
distribution: $w_{i}\sim\lambda\exp(-\lambda w_{i})$ $(w_{i}\geq0)$
and obtain~$\bar{w}=\sum_{i=1}^{d}w_{i}$. The direction~$\hat{e}$
can be generated by drawing from the $d$-dimensional standard Gaussian
distribution followed by normalization.

\paragraph{Implementation of the exponential mechanism}

The exponential mechanism requires drawing samples from a distribution
proportional to a non-negative function. Such sampling is usually
performed using Markov chain Monte Carlo (MCMC) methods~\cite{press2007numerical},
which draw samples by simulating a Markov chain whose stationary distribution
is the target distribution. In this paper, we use the Metropolis algorithm
with a multivariate Gaussian proposal distribution. Due the shape
of the constraint set, the covariance matrix~$\Sigma$ of the Gaussian
distribution is chosen to be isotropic, and its magnitude is set to
be proportional to the size of the constraint set: $ $$\Sigma=\eta cI_{d\times d}$,
where~$I_{d\times d}$ is the~$d\times d$ identity matrix, and
$\eta=0.1$. Each sample is generated by running~$5000$ MCMC steps,
after which the Markov chain is considered to have reached its stationary
distribution.

\paragraph{Number of iterations for the subgradient method}

Although Theorem~\ref{thm:dp_subgrad_subopt} clearly shows that
an optimal number of iterations exists for a given choice of~$\epsilon$,
the suboptimality bound is often loose for a given problem so that
optimizing the bound does not provide direct guidance for choosing
the number of iterations. In practice, we observe that the objective
value is quite robust to the number of iterations, as shown in Fig.~\ref{fig:dp_subgrad_typical}.
The plot also includes the objective values obtained from the regular
subgradient method, which decrease slightly as the number of iteration
grows. Due to this robustness, in all subsequent simulations, the
number of iterations is fixed at~$100$. 

\begin{figure}
\begin{centering}
\includegraphics[width=3in]{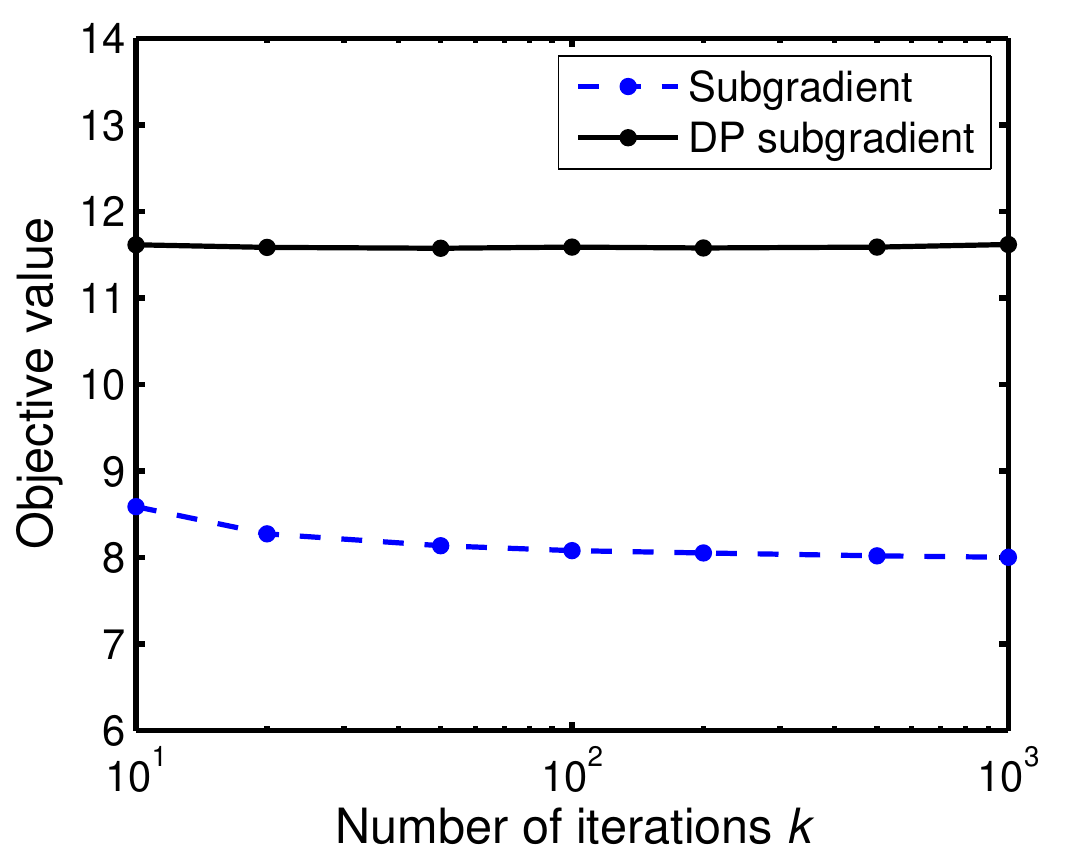}
\par\end{centering}

\caption{Objective values as a function of the number of iterations. Blue:
Objective values obtained from the regular subgradient method. Black:
Objective values obtained from the differentially private subgradient
method.}

\label{fig:dp_subgrad_typical}
\end{figure}

\subsection{Results and discussions}

The simulations investigate the effects of changing~$c$ (the size
of the constraint set~$\mathcal{P}$) and~$m$ (the number of member
affine functions) on all the mechanisms presented in Section~\ref{sec:Privacy-preserving-mechanisms}.
Fig.~\ref{fig:dp_comparison_size} shows the expected objective value
as a function of~$c$ as well as the true optimal value obtained
by solving the original problem. For all privacy preserving mechanisms,
the expected optimal value eventually grows as~$c$ increases, except
that it shows some initial decrease for the exponential mechanism
and the differentially private subgradient method. This non-monotonic
behavior can be explained by noticing two factors that contribute
to the objective value. One factor is the effect of~$c$ on the (original)
optimization problem itself. As~$c$ increases, it leads to a more
relaxed optimization problem and consequently decreases the true optimal
value (magenta dashed line). Another factor of~$c$ is on the amount
of perturbation introduced by the mechanisms. For example, for the
mechanism that perturbs the solution directly, the magnitude of the
vector Laplace noise grows with~$c$. For the exponential mechanism,
the distribution from which the solution is drawn will become less
concentrated around the optimal solution as~$c$ grows. We are unable
to provide a definitive explanation for the other two mechanisms,
but it is expected that changes in~$c$ have a similar effect. 

\begin{figure}
\begin{centering}
\includegraphics[width=3in]{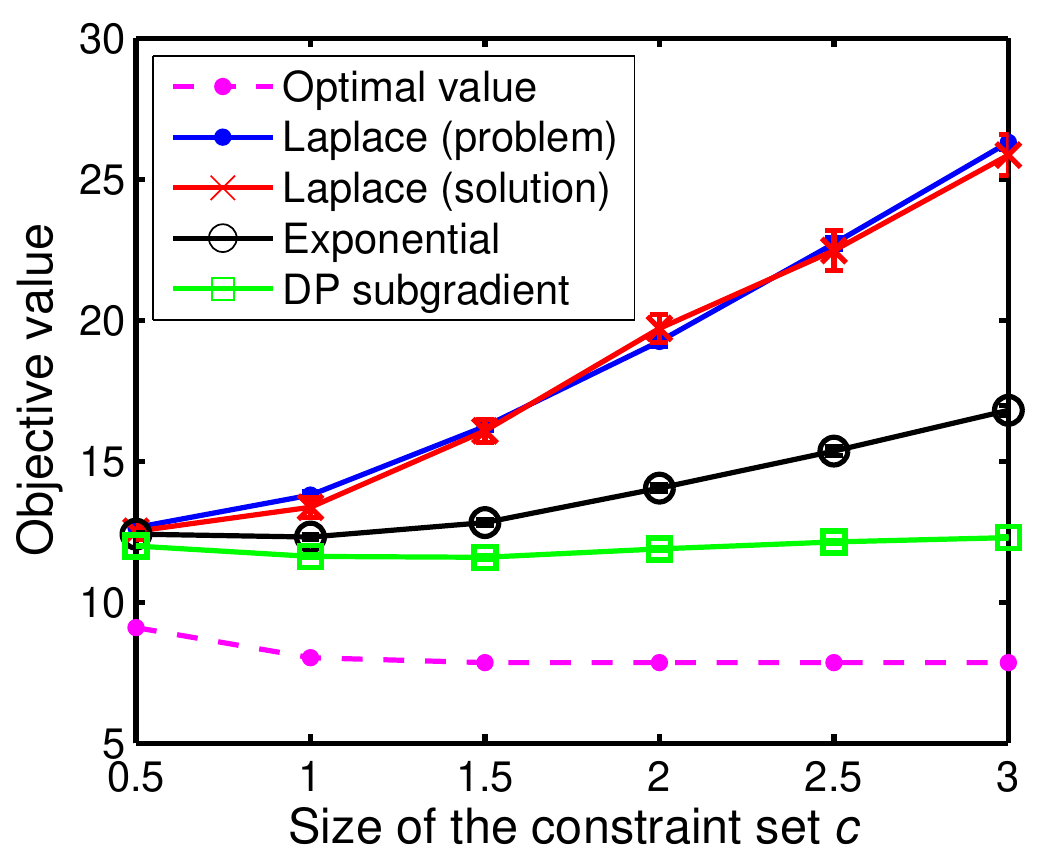}
\par\end{centering}

\caption{Objective values (with error bars corresponding to~$2\sigma$ error)
as a function of~$c$ (the size of the constraint set~$\mathcal{P}$)
for different mechanisms. The true optimal value is shown for comparison.}
\label{fig:dp_comparison_size}
\end{figure}

The effect of~$m$ on the objective value is illustrated in Fig.~\ref{fig:dp_comparison_num}.
As~$m$ increases, more affine functions will be added (i.e., the
affine functions used for a smaller~$m$ is a subset of those used
for a larger~$m$). Unlike changing~$c$, increasing~$m$ causes
the objective value to monotonically increase. First of all, adding
more affine functions causes the optimal value (magenta dashed line)
to increase even in the absence of privacy constraints. In addition,
at least for the case when the problem is perturbed, the magnitude
of the vector Laplace noise grows with~$m$. For the differentially
subgradient method, increase in~$m$ also causes the suboptimality
gap~$\bar{\gamma}$ (that has~$\log m$ dependence) to increase.

\begin{figure}
\begin{centering}
\includegraphics[width=3in]{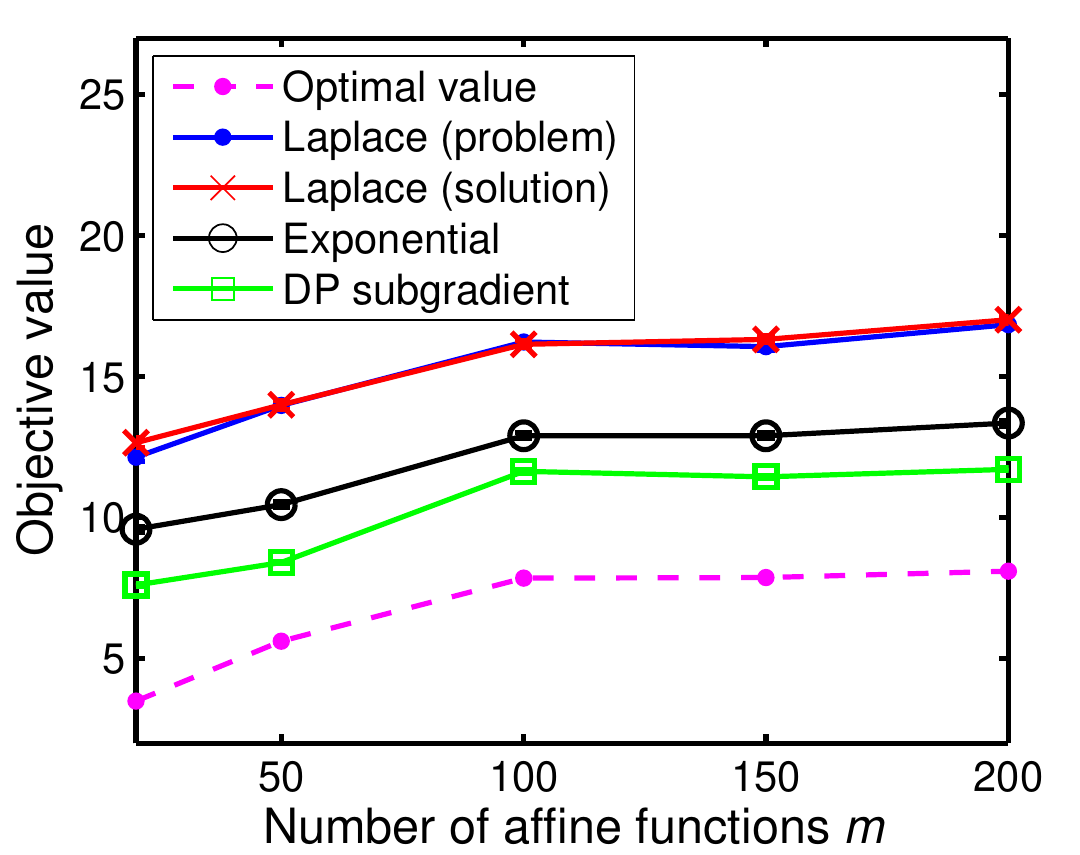}
\par\end{centering}

\caption{Objective values (with error bars corresponding to~$2\sigma$ error)
as a function of~$m$ (the number of member affine functions) for
different mechanisms. The true optimal value is shown for comparison.}
\label{fig:dp_comparison_num}
\end{figure}

As an interesting observation from all simulations, the differentially
subgradient method is superior to other mechanisms. It achieves the
lowest expected suboptimality among all mechanisms. Also, when~$c$
increases, it has the slowest growth rate of suboptimality. The reason
that why subgradient method works best is not evident from the suboptimality
analysis presented in Section~\ref{sec:Privacy-preserving-mechanisms}.
It is known that the subgradient method is quite robust to \emph{unbiased}
noise in subgradients (often known as the stochastic subgradient method).
However, in this work, the noise introduced by the exponential mechanism
in Algorithm~\ref{alg:dp_subgrad} is biased so that the analysis
on stochastic subgradient method does not directly apply. This remains
a interesting question for future investigations.

\section{Conclusions}

In this paper, we study the problem of preserving differential privacy
for the solution of convex optimization problems with a piecewise
affine objective. Several privacy-preserving mechanisms are presented,
including the Laplace mechanism\emph{ }applied on either the problem
data or the problem solution\emph{, }the\emph{ }exponential mechanism,
and the differentially private subgradient method. Theoretical analysis
on the suboptimality of these mechanisms shows the trade-offs between
optimality and privacy: more privacy can be provided at the expense
of sacrificing optimality. Empirical numerical experiments show that
the differentially private subgradient method has the least adverse
effect on optimality for a given level of privacy. In addition, it
is likely that the scheme of providing privacy by iteratively solving
an optimization problem privately (as used by the private subgradient
method) can be applied to more general convex optimization problems
beyond the specific form studied in this paper. This appears to be
an interesting direction for future investigations.

\textbf{Acknowledgments}. The authors would like to thank Aaron Roth
for providing early access to the draft on differentially private
linear programming~\cite{hsu2014privately} and helpful discussions
on differential privacy. This work was supported in part by the NSF
(CNS-1239224) and TerraSwarm, one of six centers of STARnet, a Semiconductor
Research Corporation program sponsored by MARCO and DARPA.

\bibliographystyle{abbrv}
\bibliography{ref}

\appendix

\section{Proofs}

\subsection{\label{sub:proof_sensitivity_exp_mech}Proof of Lemma~\ref{lem:sensitivity_exp_mech}}
\begin{proof}
Fix~$x$, $D$, and $D'$, and consider the quantity
\begin{align*}
\bar{\Delta}_{u} & =u(x,D')-u(x,D)\\
 & =\max_{i=1,2,\dots,m}\{a_{i}^{T}x+b_{i}\}-\max_{i=1,2,\dots,m}\{a_{i}^{T}x+b_{i}'\}.
\end{align*}
Define
\[
j=\arg\max_{i}\{a_{i}^{T}x+b_{i}\}\quad\text{and}\quad k=\arg\max_{i}\{a_{i}^{T}x+b_{i}'\}.
\]
Using the fact that $a_{j}^{T}x+b_{j}\geq a_{i}^{T}x+b_{i}$ for all
$i\in\{1,2,\dots,m\}$, we have 
\begin{align}
\bar{\Delta}_{u} & =(a_{j}^{T}x+b_{j})-(a_{k}^{T}x+b_{k}')\label{eq:utility_ineq_bound1}\\
 & \geq(a_{k}^{T}x+b_{k})-(a_{k}^{T}x+b_{k}')=b_{k}-b_{k}'.
\end{align}
Similarly, since $a_{k}^{T}x+b_{k}'\geq a_{i}^{T}x+b_{i}'$ for all
$i\in\{1,2,\dots,m\}$, we have
\begin{equation}
\bar{\Delta}_{u}\leq(a_{j}^{T}x+b_{j})-(a_{j}^{T}x+b_{j}')=b_{j}-b_{j}'.\label{eq:utility_ineq_bound2}
\end{equation}
Combining~(\ref{eq:utility_ineq_bound1}) and~(\ref{eq:utility_ineq_bound2})
together yields
\[
|\bar{\Delta}_{u}|\leq\max\{|b_{k}-b_{k}'|,|b_{j}-b_{j}'|\}.
\]
This is due to the fact that if $\alpha\leq\gamma\leq\beta$ for any
constants~$\alpha$, $\beta$, and $\gamma$, then $|\gamma|\leq\max\{|\alpha|,|\beta|\}$.
Maximizing~$|\bar{\Delta}_{u}|$ over all possible adjacent pairs
of $D$ and $D'$ yields
\begin{equation}
\max_{D,D'\colon\mathrm{Adj}(D,D')}|\bar{\Delta}_{u}|=b_{\max}.\label{eq:utility_ineq_max}
\end{equation}
Since~(\ref{eq:utility_ineq_max}) holds for any~$x$, we have
\[
\Delta_{u}=\max_{x}\max_{D,D'\colon\mathrm{Adj}(D,D')}|u(x,D)-u(x,D')|=\max_{x}\max_{D,D'\colon\mathrm{Adj}(D,D')}|\bar{\Delta}_{u}|=b_{\max}.
\]

\end{proof}

\subsection{\label{sub:proof_dp_subgrad_subopt}Proof of Theorem~\ref{thm:dp_subgrad_subopt}}
\begin{proof}
The proof follows the same procedure as the convergence proof for
the stochastic subgradient descent method (cf.~\cite{shor1998nondifferentiable}),
except for the presence of additional terms that depend on~$\mathbb{E}_{i^{*}}[\gamma(i^{*},x^{(k)})]$
due to violation of the subgradient condition~(\ref{eq:subgrad_cond}). 

At iteration~$k$, we have 
\begin{align*}
\Vert x^{(k+1)}-x_{\mathrm{opt}}\Vert_{2}^{2} & =\Vert x^{(k)}-\alpha_{k}g^{(k)}-x_{\mathrm{opt}}\Vert_{2}^{2}\\
 & =\Vert x^{(k)}-x_{\mathrm{opt}}\Vert_{2}^{2}-2\alpha_{k}g^{(k)T}(x^{(k)}-x_{\mathrm{opt}})+\alpha_{k}^{2}\Vert g^{(k)}\Vert_{2}^{2}.
\end{align*}
Take $\mathbb{E}(\cdot|x^{(k)})$ on both sides to obtain
\begin{align*}
 & \mathbb{E}\left(\left.\Vert x^{(k+1)}-x_{\mathrm{opt}}\Vert_{2}^{2}\right|x^{(k)}\right)\\
 & \qquad=\Vert x^{(k)}-x_{\mathrm{opt}}\Vert_{2}^{2}-2\alpha_{k}\mathbb{E}(g^{(k)T}(x^{(k)}-x_{\mathrm{opt}})|x^{(k)})+\alpha_{k}^{2}\mathbb{E}\left(\left.\Vert g^{(k)}\Vert_{2}^{2}\right|x^{(k)}\right).
\end{align*}
Since~$g^{(k)}$ is computed from Algorithm~\ref{alg:dp_subgrad},
we have
\begin{align*}
\mathbb{E}(g^{(k)T}(x^{(k)}-x_{\mathrm{opt}})|x^{(k)}) & \geq f(x^{(k)})-f(x_{\mathrm{opt}})-\mathbb{E}_{i^{*}}[\gamma(i^{*},x^{(k)})]\\
 & \geq f(x^{(k)})-f_{\mathrm{opt}}-\bar{\gamma}(\epsilon/k),
\end{align*}
where~$i^{*}$ is defined in Algorithm~\ref{alg:dp_subgrad}. This
leads to
\begin{align*}
 & \mathbb{E}\left(\left.\Vert x^{(k+1)}-x_{\mathrm{opt}}\Vert_{2}^{2}\right|x^{(k)}\right)\\
 & \qquad\leq\Vert x^{(k)}-x_{\mathrm{opt}}\Vert_{2}^{2}-2\alpha_{k}(f(x^{(k)})-f_{\mathrm{opt}}-\bar{\gamma}(\epsilon/k))+\alpha_{k}^{2}\mathbb{E}\left(\left.\Vert g^{(k)}\Vert_{2}^{2}\right|x^{(k)}\right).
\end{align*}
Now take expectation with respect to $x^{(k)}$:
\begin{align*}
 & \mathbb{E}\Vert x^{(k+1)}-x_{\mathrm{opt}}\Vert_{2}^{2}\\
 & \qquad\leq\mathbb{E}\Vert x^{(k)}-x_{\mathrm{opt}}\Vert_{2}^{2}-2\alpha_{k}(\mathbb{E}f(x^{(k)})-f_{\mathrm{opt}}-\bar{\gamma}(\epsilon/k))+\alpha_{k}^{2}\mathbb{E}\Vert g^{(k)}\Vert_{2}^{2}.
\end{align*}
Repeat this procedure to obtain
\begin{align*}
 & \mathbb{E}\Vert x^{(k+1)}-x_{\mathrm{opt}}\Vert_{2}^{2}\\
 & \qquad\leq\mathbb{E}\Vert x^{(1)}-x_{\mathrm{opt}}\Vert_{2}^{2}-2{\textstyle \sum_{i=1}^{k}}\alpha_{i}(\mathbb{E}f(x^{(i)})-f_{\mathrm{opt}}-\bar{\gamma}(\epsilon/k))+{\textstyle \sum_{i=1}^{k}}\alpha_{i}^{2}\mathbb{E}\Vert g^{(i)}\Vert_{2}^{2}.
\end{align*}
Rearrange the above and use the concavity of element-wise minimum
to obtain the suboptimality bound~(\ref{eq:dp_subgrad_subopt}).\end{proof}

\end{document}